\documentclass{article}
\usepackage[utf8]{inputenc}

\usepackage{graphicx,graphics}

\usepackage{rotating}
\usepackage[twoside]{geometry}
\usepackage{epsfig}
\usepackage{indentfirst}
\usepackage[usenames,dvipsnames,svgnames,table]{xcolor}
\usepackage{cmap}

\usepackage{amssymb,amsfonts,amstext,amsthm}
\usepackage{mathrsfs}
\usepackage[intlimits]{amsmath}
\newtheorem{definition}{\bf Definition}[section]
\newtheorem{proposition}{\bf Proposition}[section]

\newtheorem{theorem}{\bf Theorem}[section]
\newtheorem{corollary}{\bf Corollary}[section]
\newtheorem{remark}{Remark}[section]
\newtheorem{example}{Example}[section]
\DeclareMathOperator*{\rng}{rng}

\newcounter{ACP1}
\newenvironment{ACP1}{\refstepcounter{ACP1}\equation}{\tag{ACP}\endequation}

\title{Category theorems for Schr\"odinger semigroups}
\author{Moacir Aloisio, Silas L. Carvalho, and C\'{e}sar R. de Oliveira\thanks{Corresponding author. Telephone +55 16 3351 9153, fax +55 16 3361 2081.}}
\date{November 2019}

\begin{document}

\maketitle

\begin{abstract} Stimulated by the category theorems of Eisner and Ser\'eny in the setting of unitary and isometric $C_0$-semigroups on separable Hilbert spaces, we prove category theorems for Schr\"odinger semigroups. Specifically, we show that, to a given class of Schr\"{o}dinger semigroups, Baire generically the semigroups are strongly  stable but not exponentially stable. We also present a typical spectral property of the corresponding Schr\"{o}dinger operators. 
\end{abstract} 

\


\

\noindent{\bf  AMS classification codes}: 47D06 47D08 34D20 47A10


\section{Introduction}

\noindent A central question in the theory of differential equations refers to the asymptotic behaviour of their solutions; for instance, whether they reach an equilibrium and, if so, with which speed. This kind of question is addressed by the asymptotic theory of $C_0$-semigroups. More specifically, here we consider the theory of stability for solutions of the abstract Cauchy problem on a Hilbert space $\mathcal H$, that is,
\begin{ACP1}\label{cauchyproblem}
\begin{cases} \dot{x}(t) = A x(t), \; t \geq 0, \\ x(0) = x, \; x \in {\mathcal{H}},  \end{cases}
\end{ACP1}
where~$A$ is the generator of a $C_0$-semigroup~$(T(t))_{t\geq 0}$ on~${\mathcal{H}}$.  

We denote the resolvent set of~$A$ by $\varrho(A)$, that is, the set of all $\lambda \in {\mathbb{C}}$ for which the resolvent operator of $A$ at $\lambda$,
\[R(\lambda,A) : {\cal{H}} \longrightarrow {\cal{D}}(A), \, \, \, R(\lambda,A) := (\lambda I - A)^{-1},\]
exists and is bounded. The spectrum of $A$ is denoted by $\sigma(A) = {\mathbb{C}}\backslash \rho(A)$. 

We recall that a $C_0$-semigroup~$(T(t))_{t\geq 0}$ on ${\mathcal{H}}$ is said to be bounded if there exists a constant $C>0$ so that, for each $t \geq 0$, $\|T(t)\|_{{\mathcal{B}}({\mathcal{H}})} \leq C$; if $C=1$, then it is called a $C_0$-semigroup of contractions. 

We also recall that $(T(t))_{t\geq 0}$ is (strongly) stable if, for every $x\in \mathcal{H}$, 
\[\lim_{t \to \infty} \|T(t)x\|_{{\mathcal{H}}} =0;\] 
\noindent $(T(t))_{t\geq 0}$ is exponentially stable if there exist constants $C> 0$ and $a>0$ such that, for every $t \geq 0$,
\[\|T(t)\|_{{\mathcal{B}}({\mathcal{H}})}\leq C\, e^{-ta}.\]
 
\subsection*{Stability and spectrum}
 
\noindent Over the three last decades, the asymptotic theory of $C_0$-semigroups on Hilbert spaces had a fast development, with the solution of a large number of long-standing open problems. Among such problems, one can highlight the characterization of exponential stability for $C_0$-semigroups of contractions on Hilbert spaces, which has been proven independently by Herbst \cite{Herbst}, Howland \cite{Howland} and Pr\"{u}ss~\cite{Pruss}, 

\begin{theorem}[Gearhart-Pr\"{u}ss Theorem]\label{prusstheorem} A $C_0$-semigroup $(T(t))_{t\geq 0}$ of contractions on a Hilbert space~$\mathcal{H}$, with generator $A$, is exponentially stable if, and only if,  
\[i {\mathbb{R}} \subset \varrho(A) \quad  and \quad \limsup_{|\lambda| \to \infty} \|R(i\lambda,A)\|_{\mathcal{B}(\mathcal{H})} < \infty.\]
\end{theorem} 

Gearhart-Pr\"{u}ss Theorem relates properties of the resolvent of the generator to decaying rates of the semigroup. From the point of view of  applications, since the resolvent of the generator is often easier to compute than the semigroup, this result has been used to obtain numerous applications to PDEs. In this context, we refer to \cite{Anantharaman,Conti,Riveraa,van}, among others. We also mention the papers \cite{Batty,Borichev,Rozendaal} for modern results and additional comments about the asymptotic theory of $C_0$-semigroups. 

We note that this relation between properties of the resolvent of the generator and decaying rates of the semigroup established by Gearhart-Pr\"{u}ss Theorem can be reformulated in the case that $A$ is a negative self-adjoint operator (see Corollary \ref{prusscor} below). 

It follows from the Spectral Theorem that every negative self-adjoint operator $A$ on a Hilbert space~$\mathcal{H}$ generates a self-adjoint $C_0$-semigroup of contractions
\[e^{tA} = \int_{\sigma(A)} e^{t\lambda} \, \, \, dE^A(\lambda),\]
where $E^A$ is the resolution of the identity of $A$; so, for every $x \in \mathcal{H}$, 
\[\langle e^{tA}x,x\rangle  = \int_{\sigma(A)} e^{t\lambda} \, \, \, d\mu_x^A(\lambda),\]
 where $\mu_x^A$ denote the spectral measure of~$A$ associated with $x$. It is well known that every self-adjoint $C_0$-semigroup of contractions is of this form \cite{Rudin}.
 
 Note that it is always true that, in this case, 
\[\limsup_{|\lambda| \to \infty} \|R(i\lambda,A)\|_{\mathcal{B}(\cal{H})} = 0.\]
Hence, since $\sigma(A) \subset {\mathbb{R}_-}$,  the following result is a direct consequence of Theorem \ref{prusstheorem}.

\begin{corollary}\label{prusscor} Let $A$ be a negative self-adjoint operator. Then, $(e^{tA})_{t\geq 0}$  is exponentially stable if, and only if, $0 \not \in \sigma(A)$. 
\end{corollary}

The problem of obtaining lower bounds for the decaying rates of stable $C_0$-semigroups goes through the understanding of how the spectrum of the generator touches or approaches the imaginary axis. In this context, stimulated by the category theorems of Eisner and Ser\'eny in \cite{Eisner}, which show that the set of all weakly stable unitary groups (isometric semigroups) is of first category, while the set of all almost weakly stable unitary groups (isometric semigroups) is residual for an appropriate topology, in this paper we use Corollary \ref{prusscor} to prove category theorems for Schr\"odinger semigroups (Theorem \ref{maintheoremcategory}). To the best of our knowledge, none of this has been detailed in the literature yet.

\subsection*{Main results}

\noindent Fix $a>0$ and consider the family of negative Schr\"{o}dinger operators $H_V$,  defined on ${{\mathcal{H}}^2}({{\mathbb{R}}}^\nu)$, $\nu\in\mathbb{N}$, by the action 
\[(H_V u)(x) := \triangle u(x) + V(x)u(x),\]
with  $V \in {\cal{B}}^\infty({{\mathbb{R}}}^\nu)$ (the space of bounded Borel functions) such that, for each $x \in {\mathbb{R}}^\nu$, $-a \leq V(x) \leq 0$. Denote by $X_a^\nu$ the set of these operators endowed with the metric
\[d(H_V,H_U):=\displaystyle\sum_{j= 0}^\infty \min(2^{-j},\Vert V-U\Vert_j),\]
where $\Vert V-U\Vert_j:= \displaystyle\sup_{x\in B(0,j)}\vert V(x)-U(x)\vert$. $X_a^\nu$ is (by Tychonoff's theorem) a compact metric space so that convergence in the metric implies strong resolvent convergence (Definition \ref{convergence}). Namely, if $d(H_{V_k},H_V)\rightarrow 0$ in $X_a^\nu$ (so, for every $x \in {\mathbb{R}}^\nu$, $\displaystyle\lim_{k \to \infty} V_k(x) = V(x)$), then, for each $u \in {\mathrm L}^2({\mathbb{R}}^\nu)$, by the second resolvent identity and  dominated convergence, 
\begin{eqnarray*}\| (R_i(H_{V_k}) - R_i(H_{V}))u \|_{{\mathrm L}^2({\mathbb{R}}^\nu)} &=& \| R_i(H_{V_k})(V_k - V)R_i(H_V)u \|_{{\mathrm L}^2({\mathbb{R}}^\nu)}\\
&\leq& \|(V_k - V)R_i(H_V)u \|_{{\mathrm L}^2({\mathbb{R}}^\nu)} \longrightarrow 0
\end{eqnarray*}
as $k \rightarrow \infty$.

We shall prove the following result.

\begin{theorem}\label{maintheoremcategory}
For each $a>0$ and each $\nu\in\mathbb{N}$, 
\[\{H \in X_a^\nu \mid (e^{tH})_{t\geq 0} \text{ \ is stable but not exponentially stable}\}\]
is a dense $G_\delta$ set in~$X_a^\nu$.
\end{theorem}

It is natural to investigate the asymptotic behaviour of the orbits of the semigroups generated by the operators in the typical set in Theorem~\ref{maintheoremcategory}. The next result, which is a particular case of Theorem~1.2  in~\cite{Aloisio}, says something in this direction.

\begin{theorem}[Theorem 1.2 in \cite{Aloisio}]\label{ACOtheorem}Let~$A$ be a negative self-adjoint operator in~$\mathcal H$  and  $\alpha,\beta:{\mathbb{R}}_+ \longrightarrow (0,\infty)$ functions so that, for each $\epsilon>0$,
\[\lim_{t \to \infty} \alpha(t) = \infty \;\; and \;\; \lim_{t \to \infty} \beta(t)e^{-t \epsilon }= 0.\]
 Suppose that $(e^{tA})_{t\geq 0}$ is stable but is not exponential stable. Then, 
\[{\mathcal{G}}_A(\alpha,\beta) :=\{x \displaystyle\mid\limsup_{t \to \infty} \alpha(t) \Vert e^{tA}x\Vert_{\mathcal{H}} = \infty \;\; and \;\; \liminf_{t \to \infty} \beta(t) \Vert e^{tA}x\Vert_{\mathcal{H}} = 0\}\]
is a dense $G_\delta$ set in $\mathcal{H}$. Moreover, the assumption on $\beta$ is optimal, that is, $\beta$ cannot grow faster than sub-exponentially. 
\end{theorem}

\begin{remark}\label{remarkmaintheorem}{\rm  Theorem \ref{ACOtheorem} can be seen as an expression of the fact that there exist dense $G_\delta$ sets of initial values $x \in {\cal{H}}$ such that the orbit $(e^{tN}x)_{t \geq 0}$ contains a sequence that decays to zero no faster than a fixed but arbitrarily slow rate, and a sequence that decays to zero at a sub-exponential rate. To put this work into perspective, we note that Theorem \ref{maintheoremcategory} makes clear that the assumptions in the statement of Theorem \ref{ACOtheorem} are very natural.}
\end{remark}

\begin{remark}{\rm It follows from Theorems \ref{maintheoremcategory} and \ref{ACOtheorem}  that, for each $a>0$ and each $\nu\in\mathbb{N}$, typically in $X_a^\nu$, the orbits of each Schr\"odinger semigroup $(e^{tH_V})_{t\geq 0}$, typically in ${\mathrm L}^2({\mathbb{R}}^\nu)$, have decaying rates depending on subsequences of time going to infinity. Hence, for every $X_a^\nu$, the dynamics is typically (from the topological viewpoint) nontrivial.}
\end{remark}

\noindent In this work, we are also particularly interested in using Theorem \ref{maintheoremcategory} to say something about the
spectral properties of this class of Schr\"odinger semigroups. 

\begin{definition}{\rm  Let $\mu$ be a finite (positive) Borel measure on $\mathbb{R}$. The pointwise lower and upper scaling exponents of $\mu$  at $w \in \mathbb{R}$ are defined, respectively, as  
\[d_\mu^-(w) := \liminf_{\epsilon \downarrow 0} \frac{\ln \mu (B(w,\epsilon))}{\ln \epsilon} \quad{\rm  and }\quad d_\mu^+(w) := \limsup_{\epsilon \downarrow 0} \frac{\ln \mu (B(w,\epsilon))}{\ln \epsilon},\]
if, for all $\epsilon>0$,  $\mu(B(w,\epsilon))> 0$; if not, $d_\mu^{\mp}(w) := \infty$; here, $B(w,\epsilon)$ stands for the open ball of radius $\epsilon$ centered at $w$.} 
\end{definition}

\begin{proposition}[Proposition in 2.1 in \cite{Aloisio}]\label{decpolproposition}  Let $A$ be a negative self-adjoint operator. Then, for every $x \in \cal{H}$, $x \not = 0$,
\begin{equation*}
\liminf_{t \to \infty} \frac{\ln  \|e^{tA}x\|_{{\mathcal{H}}}^2}{\ln t} = -d_{\mu_x^A}^+(0),
\end{equation*}
\begin{equation*}
\limsup_{t \to \infty} \frac{\ln \|e^{tA}x\|_{{\mathcal{H}}}^2}{\ln  t} = -d_{\mu_x^A}^-(0).
\end{equation*}
\end{proposition}

Proposition \ref{decpolproposition} relates, for each $x\in\mathcal{H}$, the polynomial decaying rates of $\|e^{tA}x\|_{{\mathcal{H}}}$ to dimensional properties of  the spectral measure $\mu_x^A$ of~$A$ associated with $x$. We note that this result establishes an explicit relation between the dynamics of the semigroup and the spectral properties of its generator. 

The next result is a direct application of Proposition \ref{decpolproposition} and Theorem \ref{maintheoremcategory}.

\begin{theorem}\label{spectraltheorem} For each $a>0$ and each $\nu\in\mathbb{N}$, there exists a dense $G_\delta$ set~$G_a^\nu$ in ${\mathrm L}^2({\mathbb{R}}^\nu)$, such that, for every $f \in G_a^\nu$,  
\[J_a^\nu(f) := \{H \in X_a^\nu  \mid d_{\mu_f^H}^-(0) = 0  \ and \ d_{\mu_f^H}^+(0) = \infty\}\]
is a dense $G_\delta$ set in $X_a^\nu$, where $\mu_f^H$ represents the spectral measure of~$H$ associated with $f$.
\end{theorem}

\begin{remark}{\rm  Theorem \ref{spectraltheorem} indicates the subtlety of the relations between the dynamics of a Schr\"{o}dinger semigroup and the spectral properties of its generator; that is, for each initial condition~$f \in {\mathrm L}^2({\mathbb{R}}^\nu)$ in a dense $G_\delta$ set, the generic situation in $X_a^\nu$ is characterized by extreme values of the scaling exponents of spectral measures (i.e., $d_{\mu_f^H}^-(0) = 0, d_{\mu_f^H}^+(0) = \infty$) and the corresponding dynamical consequences described in Proposition~\ref{decpolproposition}.}
\end{remark}

\begin{remark}{\rm Although important, it is not always easy to present a typical spectral property for a family of Schr\"{o}dinger operators. Particularly for $X_a^\nu$, a difficulty that may arise is that the spectrum of the operators are Baire typically purely singular continuous (this has been proven in \cite{SimonAnn}). In this context, in order to get around this, we use in the proof of Theorem \ref{spectraltheorem} an argument involving separability that, in a sense, connects the typical behaviour of ${\mathrm L}^2({\mathbb{R}}^\nu)$ with the typical behaviour in $X_a^\nu$.}
\end{remark}

\ 

The paper is organized as follows. In Section \ref{secproofmaintheorem}, we present the proof of Theorem \ref{maintheoremcategory}. In Section \ref{secapplic}, we prove Theorem \ref{spectraltheorem}

\

Some words about notation: $\mathcal{H}$ denotes a separable complex Hilbert space and $\mathcal{B}(\mathcal{H})$  the space of all bounded linear operators on~$\mathcal{H}$. If~$A$ is a linear operator in~$\mathcal{H}$, we denote its domain by ${\mathcal{D}}(A)$ and its range by $\rng A$. 
For each set $\Omega \subset {\mathbb{R}}^\nu$, $\nu \in \mathbb{N}$, $\chi_\Omega$ denotes the characteristic function of the set $\Omega$. For each $x \in {\mathbb{R}}^\nu$ and $r>0$,  $B(x,r)$ denotes the open ball (with respect to the Euclidean distance) of radius~$r$  centered  on~$x$. Finally, $\| \cdot \|_{\mathcal{B}(\mathcal{H})}$ denotes the norm in $\mathcal{B}(\mathcal{H})$ and $\| \cdot \|_{\mathcal{H}}$ denotes the norm in~$\mathcal{H}$.


\section{Proof of Theorem \ref{maintheoremcategory}}\label{secproofmaintheorem}

\noindent Some preparation is required in order to present the proof of Theorem \ref{maintheoremcategory}. 

\begin{definition}{\rm A sequence of bounded linear operators $(A_n)$ strongly converges to $A$ on~$\cal{H}$ if, for every $x \in \cal{H}$, $A_n x \longrightarrow Ax$ in $\cal{H}$.}
\end{definition}

Now we present precise definitions of a sequence of (unbounded) negative self-adjoint operators~$(A_n)$ approaching another one~$A$.

\begin{definition}\label{convergence}{\rm Let $(A_n)$ be a sequence of negative self-adjoint operators, that is, $A_n \leq 0$, and let $A$ be another negative self-adjoint operator. One says that: 
\begin{enumerate}
\item [i)] $A_n$ converges to $A$,  as $n \rightarrow \infty$, in the strong resolvent sense (SR) if $R(i,A_n)$  strongly converges to $R(i,A)$.
\item [ii)] $A_n$ converges to $A$, as $n \rightarrow \infty$, in the strong dynamical sense (SD) if, for each $t \geq 0$, $e^{tA_n}$ strongly converges to $e^{tA}$.
\end{enumerate}}
\end{definition}

\begin{proposition}[Proposition 10.1.9 in \cite{Oliveira}]\label{SR-SDtheorem} The SR convergence of negative self-adjoint operators implies in the SD convergence.
\end{proposition}

We recall that a metric space $(X,d)$ of negative self-adjoint operators, acting on~$\mathcal H$, is called {\em regular} \cite{SimonAnn} if it is complete and convergence with respect to~$d$ implies strong resolvent convergence of operators.

\begin{proposition}\label{mainproposition}
Let $(X,d)$ be a regular space of negative self-adjoint operators. Suppose that 
\begin{enumerate}
\item[{\rm i)}] $\{A \in X \mid 0 \in \sigma(A)\}$ is dense in~$X$,
\item[{\rm ii)}] $\{A \in X \mid 0 \not \in \sigma(A)\}$ is dense in~$X$.
\end{enumerate}
Then, 
\[\{A \in X \mid (e^{tA})_{t\geq 0} \text{ is stable but not exponentially stable}\}\]
is a dense $G_\delta$ set in~$X$.
\end{proposition}

\begin{proof} To prove Proposition \ref{mainproposition}, it is enough to show that

\begin{enumerate}
\item [i)] $E := \{A \in X \mid   (e^{tA})_{t\geq 0}$  is exponentially stable$\}$ is meager in~$X$,
\item [ii)] $Y := \{A \in X \mid   (e^{tA})_{t\geq 0}$  is stable$\}$ is a dense $G_\delta$ set in~$X$.
\end{enumerate}

Firstly, let us show that $E$ is an $F_\sigma$ set in~$X$. It follows from Proposition \ref{SR-SDtheorem} that each section of the mapping  
\[{\mathbb{R}}_+ \times {\mathcal{H}} \times X \ni (t,x,A) \mapsto  \Vert e^{tA}x\Vert_{\mathcal{H}}\]
is continuous. Thus, for each $n \geq 1$, the mapping 
\[X \ni A \mapsto \sup_{t \geq 0} \sup_{\Vert x\Vert_{\mathcal{H}} = 1}  e^{\frac{t}{n}}\Vert e^{tA}x\Vert_{\mathcal{H}} = \sup_{t \geq 0} e^{\frac{t}{n}}\Vert e^{tA}\Vert_{{\cal{B}}(\mathcal{H})}\] 
is lower semicontinuous, from which it follows that, for each $n \geq 1$, the set
\[F_n = \{ A\in X \mid \sup_{t \geq 0} e^{\frac{t}{n}}\Vert e^{tA}\Vert_{{\cal{B}}(\mathcal{H})} \leq 1 \}\]
is closed.

Since the inclusion $\bigcup_{n \geq 1} F_n \subset E$ is immediate, we just need to prove that $E\subset\bigcup_{n \geq 1} F_n$. Let $A \in E$; then, by  Corollary \ref{prusscor}, one has that $a := -\sup \sigma(A)>0$. Nevertheless,  $\sup_{t \geq 0} e^{ta} \Vert e^{tA}\Vert_{{\cal{B}}(\mathcal{H})} \leq 1$, from which it follows that  $A \in \bigcup_{n \geq 1} F_n$. Thus, $E$ is an $F_\sigma$ in~$X$. 

Now, since $\{ A \in X \mid  0 \in \sigma(A)\}$ is dense in~$X$, it follows from Corollary \ref{prusscor} that $E^c$ is dense in~$X$; therefore, $E$ is meager in~$X$ and $i)$ is proven. 

It remains to prove $ii)$. By the previous arguments, given $x \in \mathcal{H}$, for each $k\geq 1$ and each  $n\geq 1$, 
\[\bigcup_{t\geq k} \big\{A \in X \mid  \Vert e^{tA}x\Vert_{\mathcal{H}}< \frac{1}{n}  \big\}\]
is open, hence
\begin{eqnarray*}
Y_x:= \{A \in X \mid\lim_{t \to \infty}  \Vert e^{tA}x\Vert_{\mathcal{H}} = 0\} &=& \{A \in X \mid\liminf_{t \to \infty}  \Vert e^{tA}x\Vert_{\mathcal{H}} = 0\}\\ &=& \bigcap_{n\geq 1} \bigcap_{k \geq 1} \bigcup_{t\geq k} \big\{A \in X \ | \  \Vert  e^{tA}x\Vert_{\mathcal{H}} < \frac{1}{n}\big\}
\end{eqnarray*}
is a $G_\delta$ set in~$X$. Since, by Corollary \ref{prusscor}, $\{ A \in X \mid  0 \not  \in \sigma(A)\} \subset Y_x$ is dense in~$X$, it follows that $Y_x$ is a dense $G_\delta$ set in~$X$. 

Finally, let $\{x_k : k \geq 1\}$ be a dense subset in $\mathcal{H}$ (which is separable). Then,
\[Y = \bigcap_{k\geq 1} Y_{x_k}.\]
Namely, the inclusion  $Y \subset \bigcap_{k\geq 1} Y_{x_k}$ is obvious, and the reciprocal one follows from the fact that, for each $A\in\bigcap_{k\geq 1} Y_{x_k}$ and each $x \in \mathcal{H}$, by the Moore-Osgood Theorem, 
\[\lim_{t \to \infty}  \Vert e^{tA}x\Vert_{\mathcal{H}} = \lim_{k \to \infty}   \lim_{t \to \infty}  \Vert e^{tA}x_k\Vert_{\mathcal{H}} =0.\]
Thus, by Baire's Theorem, $Y$ is a dense $G_\delta$ set in~$X$  and $ii)$ is proven.
\end{proof}

\begin{proof}[{Proof} {\rm (Theorem~\ref{maintheoremcategory})}] By Proposition \ref{mainproposition}, we just need to show that, for every $a>0$ and every $\nu\in\mathbb{N}$, 
\begin{enumerate}
\item[i)] $\{ H_V \in X_a^\nu \mid 0 \in \sigma(H_V)\}$ is dense in $X_a^\nu$,
\item[ii)] $\{ H_V \in X_a^\nu \mid 0 \not \in \sigma(H_V)\}$ is dense in $X_a^\nu$.
\end{enumerate}

Let $H_V \in X_a^\nu$ and define, for each $k \geq 1$, $V_k := \chi_{B(0,k)} V$. Then, by Weyl's criterion (Corollary 11.3.6 in~\cite{Oliveira}), the essential spectrum of~$H_{V_k}$ is given by  $\sigma_{\mathrm{ess}}(H_{V_k}) = (-\infty,0]$; moreover, $H_{V_k} \rightarrow H_V$ in $X_a^\nu$. Thus, $\{ H_V \in X_a^\nu \ | \ 0 \in \sigma(H_V)\}$ is dense in $X_a^\nu$.

Now, let $(H_{V_l})$ be a sequence in $X_a^\nu$ such that, for each $l \geq 1$, 
\[V_l:= \frac{l}{l+1}V - \frac{a}{l+1}.\] 
It is clear that, for each $l\geq 1$, $0 \not \in \sigma(H_{V_l})$. Moreover, $H_{V_l} \rightarrow H_V$ in $X_a^\nu$. Therefore, $\{ H_V \in X_a^\nu \mid 0 \not \in \sigma(H_V)\}$  is dense in $X_a^\nu$.  
\end{proof}


\section{Proof of Theorem \ref{spectraltheorem}}\label{secapplic}

\noindent In order to prove Theorem \ref{spectraltheorem}, we need the following corollary, a direct consequence of  Proposition \ref{decpolproposition} and Theorem~\ref{ACOtheorem}.

\begin{corollary}\label{cortheoremnormal} 
Let $A$ be as in the statement of Proposition~\ref{decpolproposition}. Suppose that $(e^{tA})_{t\geq 0}$ is stable but is not exponential stable. Then, 
\[\{x\in\mathcal H \mid d_{\mu_x^{A}}^-(0) = 0  \ and \ d_{\mu_x^{A}}^+(0) = \infty\}\]
is a dense $G_\delta$ set in $\mathcal{H}$.
\end{corollary} 

\begin{proof}[{Proof} {\rm (Theorem~\ref{spectraltheorem})}]  Since, by Proposition \ref{decpolproposition},
\begin{eqnarray*} 
J_a^\nu(f) &=&\{H  \mid d_{\mu_f^H}^-(0) = 0 \ \textrm{and} \ d_{\mu_f^H}^+(0) = \infty\}\\ &=& \bigcap_{n \geq 1}\big\{ H \mid \limsup_{t \to \infty} t^{1/n} \Vert e^{tH}f\Vert_{{\mathrm L}^2({\mathbb{R}}^\nu)} = \infty \;\; \textrm{and} \;\; \liminf_{t \to \infty} t^{n} \Vert e^{tH}f\Vert_{{\mathrm L}^2({\mathbb{R}}^\nu)} = 0\big\},
\end{eqnarray*}
it follows from the arguments presented in the proof of Proposition~\ref{mainproposition} that, for every $f \in {\mathrm L}^2({\mathbb{R}}^\nu)$, $J_a^\nu(f)$ is a $G_\delta$ set in $X_a^\nu$.   

Now, let
\[C_a^\nu = \{H\in X_a^\nu \mid (e^{tH})_{t\geq 0} \text{ is stable but not exponentially stable}\}.\]
It follows from Theorem \ref{maintheoremcategory} that $C_a^\nu$ is a dense $G_\delta$ set in $X_a^\nu$. Let $\{ H_k : k \geq 1 \}$ be an enumerable dense subset of $C_a^\nu$ (which is separable, since $X_a^\nu$ is separable). So, by Corollary \ref{cortheoremnormal},
\[G_a^\nu := \bigcap_{k \geq 1} \{f  \mid d_{\mu_f^{H_k}}^-(0) = 0  \ \textrm{and} \ d_{\mu_f^{H_K}}^+(0) = \infty\}\]
is a dense $G_\delta$ set in ${\mathrm L}^2({\mathbb{R}}^\nu)$. Nevertheless, for every $f \in G_a^\nu$, $J_a^\nu(f) \supset \{ H_k : k \geq 1 \}$ is a dense $G_\delta$ set in $X_a^\nu$. 
\end{proof}

The next example, together with Theorem \ref{spectraltheorem}, says that for each $f \in G_a^\nu$ and each $H \in J_a^\nu(f)$ we have $f \not \in \rng H$, or equivalently, the partial differential equation $H u = f$ has no solution in ${\cal{D}}(H)$.

\begin{example}\label{examplerng}{\rm
Let~$A$ be a negative self-adjoint operator. Then, for $u \in \rng A$, $d_{\mu_u^{A}}^{\mp}(0) \geq 2$.}
\end{example}

Example \ref{examplerng} can be seen as a statement from the fact that every spectral measure associated with every vector of the range of a negative self-adjoint operator has a certain local regularity with respect to the Lebesgue measure. We note that this is a direct consequence of Propositions \ref{decpolproposition} and~\ref{scarefproposition}. 
\begin{proposition}\label{scarefproposition} Let~$A$ be a negative self-adjoint operator. Then, for $u \in \rng A$ and every $x \in A^{-1} \{ u \}$, one has, for each $t>0$,
\begin{equation*}
\|e^{tA}u\|_{\mathcal{H}} \leq \  \frac{\Vert x \Vert_{\mathcal{H}}}{e\;t} \,.
\end{equation*}
\end{proposition}

\begin{proof} Let $x \in A^{-1}\{u\}$. Then, by the Spectral Theorem,  for each $t>0$, 
\begin{eqnarray*}
 t^2 \|e^{tA}u\|_{\mathcal{H}}^2 &=& t^2 \|Ae^{tA}x\|_{\mathcal{H}}^2 = \int_{-\infty}^0  \ (ty)^2 \ e^{2ty}\,  {\mathrm d}\mu_x^A(y)\\ &\leq& \frac{1}{e^2}\int_{-\infty}^0  1 \, {\mathrm d}\mu_x^A(y) = \frac{\Vert x \Vert_{\mathcal{H}}^2}{e^2}.
\end{eqnarray*}
\end{proof} 

\begin{remark}
\end{remark}
{\rm \begin{enumerate}
\item[i)] We note that the polynomial decaying rate obtained in Proposition~\ref{scarefproposition} is optimal. Indeed, define $M : {\mathcal D}(M)\subset {\mathrm L}^2[0, \infty) \longrightarrow {\mathrm L}^2[0, \infty)$ by 
\[(Mu)(y) = -y u(y),\] where $u\in{\mathcal{D}}(M) := \{u \in {\mathrm L}^2[0, \infty) \ | \  yu \in {\mathrm L}^2[0, \infty) \}.$
Consider $\frac{1}{2}<\delta<1$, and then define $f_\delta :[0,\infty) \to\mathbb{R}$ by the action $f_\delta(y) =  \chi_{[0,1]}  y^\delta$; $f_\delta$ clearly belongs to $\rng M$. Moreover, for every $0<\epsilon\leq 1$,
\[\mu_{f_\delta}^M(B(0,\epsilon))  = \int_0^{\epsilon} |f_\delta|^2 dy = \int_0^{\epsilon} y^{2\delta} dy = \frac{\epsilon^{2\delta+1}}{2\delta + 1}.\]
Thus, by Proposition~\ref{decpolproposition}, $\|e^{tM}f_\delta\|_{\mathrm{L}^2[0,\infty)} \geq  C_{f_\delta} t^{-1/2-\delta}$, for all $t \geq 1$, where   $C_{f_\delta}$ is a constant depending only on $f_\delta$. 
\item[ii)] Let $A$ be as in the statement of Proposition \ref{scarefproposition}. If $a =  - \sup \sigma(A) > 0$, then $(e^{tA})_{t \geq 0}$ is exponentially stable. Actually,  $\|e^{tA}\|_{\mathcal{B}(\mathcal{H})} = O(e^{-ta})$, which implies that, for each $x \in \mathcal{H}$, $\|e^{tA}x\|_{\mathcal{H}} = O(e^{-ta})$. Proposition~\ref{scarefproposition} presents more information about the decay of $\|e^{tA}u\|_{\mathcal{H}}$  in  case  $u\in\rng(A+a\textbf{1})$, since, in this case, it shows that there exists $C_u>0$, depending only on~$u$, such that, for every $t>0$,
\[\|e^{tA}u\|_{\mathcal{H}} \le C_u\, \frac{e^{-ta}}{t}.\] 
 Namely, let $t \in {\mathbb{R}}$ and $v \in \mathcal{H}$; then,
\[\| v \|_{{\mathcal{H}}}  =\| e^{-ta{\bf 1}} e^{ta{\bf 1}} v \|_{{\mathcal{H}}} \leq \|e^{-ta{\bf 1}}\|_{\mathcal{B}(\mathcal{H})} \|  e^{ta{\bf 1}} v \|_{{\mathcal{H}}}  \leq e^{-ta} \|  e^{ta{\bf 1}} v \|_{{\mathcal{H}}}.\]
If $u\in\rng(A+a\textbf{1})$ and $x \in {\cal{D}}(A)$ is such that $(A+a\textbf{1})x =u$, then, by Proposition \ref{scarefproposition},
\[\Vert x \Vert_{\mathcal{H}} \  \frac{1}{e\,t} \geq \|e^{t(A+a\textbf{1})}u \|_{{\mathcal{H}}} = \|e^{ta{\bf 1}} e^{tA}u\|_{{\mathcal{H}}} \geq e^{ta} \|e^{tA}u\|_{{\mathcal{H}}}.\]
\end{enumerate}}


\begin{center} \Large{Acknowledgments} 
\end{center}

M.A. was supported by CAPES (a Brazilian government agency);  S.L.C. thanks the partial support by FAPEMIG (Universal Project 001/17/CEX-APQ-00352-17); and C.R.dO. thanks the partial support by CNPq (contract 303503/2018-1). The authors are grateful to Pedro T. P. Lopes for fruitful discussions and helpful remarks.


\noindent  Email: moacir@ufam.edu.br, Departamento de Matem\'atica, UFAM, Manaus, AM, 369067-005 Brazil

\noindent  Email: silas@mat.ufmg.br, Departamento de Matem\'atica, UFMG, Belo Horizonte, MG, 30161-970 Brazil

\noindent  Email: oliveira@dm.ufscar.br,  Departamento  de  Matem\'atica,   UFSCar, S\~ao Carlos, SP, 13560-970 Brazil

\end{document}